\documentclass[11pt,reqno]{amsart}
\usepackage{amsthm}
\usepackage{amsmath,amssymb,graphicx,amscd,xy}
\usepackage[usenames,dvipsnames]{color}
\usepackage{graphicx}
\usepackage{longtable}
\usepackage{mathtools}
\usepackage{color}
\usepackage{verbatim}
\usepackage{pict2e}
\usepackage[utf8,utf8x]{inputenc}
\usepackage{enumitem}
\usepackage{longtable}
\usepackage{hyperref}

\parskip=\smallskipamount

\allowdisplaybreaks

\begin{document}

\author{Lars Simon}
\address{Lars Simon, Department of Mathematical Sciences, Norwegian University of Science and Technology, Trondheim, Norway}
\email{lars.simon@ntnu.no}

\author{Berit Stens\o nes}
\address{Berit Stens\o nes, Department of Mathematical Sciences, Norwegian University of Science and Technology, Trondheim, Norway}
\email{berit.stensones@ntnu.no}

\thanks{The second author is supported by the Research Council of Norway, Grant number 240569/F20.}
\thanks{Part of this work was done during the international research program "Several Complex Variables and Complex Dynamics" at the Centre for Advanced Study at the Academy of Science and Letters in Oslo during the academic year 2016/2017.}
\title{An Example on $s$-H-Convexity in $\mathbb{C}^2$}

%
%

\subjclass[2010]{Primary 32F17, 32T99.  Secondary 32T20.}
\keywords{$s$-H-convexity, worm domain, Stein neighborhood basis.}

\begin{abstract}
We construct a bounded domain $\Omega$ in $\mathbb{C}^2$ with boundary of class $\mathcal{C}^{1,1}$, such that $\overline{\Omega}$ has a Stein neighborhood basis, but is {\emph{not}} $s$-H-convex for any real number $s\geq{1}$.
\end{abstract}

\maketitle

\section{introduction}

The notion of $s$-H-convexity was introduced by J.\ Chaumat and A.-M.\ Chollet in \cite{MR941631} and goes back to work by A.\ Dufresnoy \cite{MR526786}. Given a real number $s\geq{1}$, a compact set $\emptyset\neq{}K\subseteq\mathbb{C}^n$ is called $s$-H-convex, if there exists a $C>0$ with $C\leq{1}$, such that for all $\epsilon$, $0<\epsilon\leq{1}$, there exists an open pseudoconvex subset $\Omega_{\epsilon}$ of $\mathbb{C}^n$ satisfying
\begin{align*}
\{z\in\mathbb{C}^n\colon{}d(z,K)<C\epsilon^s\}\subseteq\Omega_{\epsilon}\subseteq{}\{z\in\mathbb{C}^n\colon{}d(z,K)<\epsilon\}\text{,}
\end{align*}
where $d(\cdot{},K)$ denotes the Euclidean distance to $K$.

J.\ Chaumat and A.-M.\ Chollet obtain various $\overline{\partial}$-results for such sets, see e.g.\ \cite{MR941631}, \cite{MR1207862} and \cite{MR1190977}. Another result in that spirit is due to A.-M.\ Chollet \cite{MR1228871}.

Furthermore, the notion of $s$-H-convexity is related to the Mergelyan property. Specifically, there exists a ${k_0}(s,n)>0$, such that $\mathcal{O}(\overline{\Omega})$ is dense in $\mathcal{C}^{k}(\overline{\Omega})\cap\mathcal{O}({\Omega})$, whenever $k$ is an integer $\geq{k_0}(s,n)$ and $\Omega\subseteq\mathbb{C}^n$ is a bounded pseudoconvex domain, satisfying suitable assumptions, whose closure is $s$-H-convex.

Given these $\overline{\partial}$-results and the connection to the Mergelyan property, it becomes desirable to identify sets which are $s$-H-convex for some $s\geq{1}$. Specifically, given a bounded (pseudoconvex) domain in $\mathbb{C}^n$ whose closure admits a Stein neighborhood basis, one can ask under which additional assumptions said closure is necessarily $s$-H-convex for some $s\geq{1}$.

To our knowledge, it is unknown whether there exists a bounded (pseudoconvex) domain $\Omega$ in $\mathbb{C}^2$ with boundary of class $\mathcal{C}^{2}$ (or $\mathcal{C}^{\infty}$), such that $\overline{\Omega}$ has a Stein neighborhood basis, but is {\emph{not}} $1$-H-convex. In this paper we show that, if the smoothness assumption on the boundary is relaxed appropriately, there exists a bounded domain whose closure admits a Stein neighborhood basis, but is not $s$-H-convex for {\emph{any}} $s\geq{1}$. This is achieved by modifying the construction of the classical Diederich-Forn{\ae}ss worm domain \cite{MR0430315}. A precise statement of the main result of this paper goes as follows:

\theoremstyle{plain}
\newtheorem{oberwurmmaintheorem}[propo]{Theorem}
\begin{oberwurmmaintheorem}
\label{oberwurmmaintheorem}
There exists a bounded (pseudoconvex) domain $\Omega\neq\emptyset$ in $\mathbb{C}^2$ with boundary of class $\mathcal{C}^{1,1}$, such that:
\begin{itemize}
\item{$\overline{\Omega}$ has a Stein neighborhood basis,}
\item{$\overline{\Omega}$ is {\emph{not}} $s$-H-convex for any real number $s\geq{1}$.}
\end{itemize}
\end{oberwurmmaintheorem}

This paper is organized as follows: in Section \ref{wormpaperprelimsection} we introduce some notation, define the domain $\Omega$ from Theorem \ref{oberwurmmaintheorem} and give an informal description of our constructions. In Section \ref{wormpapershconvsection} we show that $\overline{\Omega}$ is not $s$-H-convex for any $s\geq{1}$ and in Section \ref{wormpapersteinnhbdsection} we construct a Stein neighborhood basis for $\overline{\Omega}$. Finally, in Section \ref{wormpapertechnlemmasection}, we prove the remaining lemmas from Section \ref{wormpapersteinnhbdsection}.

\section{Preliminaries}\label{wormpaperprelimsection}

From now on we let the function $g\colon\mathbb{R}\to\mathbb{R}$ be given by
\begin{align*}
\begin{split}
x\mapsto
  \begin{cases}
        {0} & \text{if } x\leq{0}\text{,}\\
        {\exp{(-1/x)}} & \text{if } x>0\text{,}
  \end{cases}
\end{split}
\end{align*}
and fix a function $\mathcal{S}\colon\mathbb{R}\to\mathbb{R}$ as well as real numbers $0<{\alpha}<{\beta}<{\pi}/2$ with the following properties:
\begin{enumerate}
\item{$\mathcal{S}$ is of class $\mathcal{C}^{\infty}$ on $\mathbb{R}\setminus\{0,{\pi}\}$ and of class $\mathcal{C}^{1,1}$ on neighborhoods of $0$ and $\pi$ respectively,}
\item\label{wormpapersissymmy}{$\mathcal{S}$ is concave on $\mathbb{R}$ and satisfies $\mathcal{S}(x+{\pi}/2)=\mathcal{S}(-x+{\pi}/2)$ for all $x\in\mathbb{R}$,}
\item\label{wormpaperprelimprop3}{$\mathcal{S}$ is $\leq{1}$ on $\mathbb{R}$ and $\equiv{1}$ on $[0,{\pi}]$,}
\item\label{wormpaperprelimprop4}{$0<\mathcal{S}<1$ on $({\pi},{\pi}+{\beta})$ and $\mathcal{S}<0$ on $({\pi}+{\beta},{\infty})$,}
\item\label{wormpaperprelimprop5}{$\mathcal{S}$ is decreasing on $[{\pi},{\pi}+{\beta}]$ and $\mathcal{S}'({\pi}+{\beta})<0$,}
\item\label{wormpaperprelimprop6}{${\alpha}<1/(4{\pi})$ and the following inequalities hold for $x\in{}[-{\alpha},{\alpha}]$:
\begin{itemize}
\item{$|\sin{(x)}-x|\leq{|x|}^3$,}
\item{$|\sin{(x)}|\geq{(3/4)\cdot{|x|}}$,}
\item{$|\tan{(x)}|\leq{2|x|}$,}
\end{itemize}
}
\item\label{wormpaperprelimprop7}{for all $x\in{[{\pi},{\pi}+{\alpha}]}$ we have:
\begin{align*}
\sqrt{\mathcal{S}(x)}={{\cos{(x-{\pi})}-g(x-{\pi})}}\text{.}
\end{align*}
}
\end{enumerate}
The existence of $\mathcal{S}$, $\alpha$ and $\beta$ with these properties is clear. Using this, we define a function
\begin{align*}
{\rho}\colon{\left({\mathbb{C}\setminus\{0\}}\right)\times\mathbb{C}} & \to\mathbb{R}\text{,}\\
(z,w) & \mapsto\left|{w-\exp{(i\cdot\ln{(|z|^2)})}}\right|^2-\mathcal{S}(\ln{(|z|^2)})\text{,}
\end{align*}
and a set
\begin{align*}
{\Omega}=\left\{(z,w)\in\mathbb{C}^2\colon{z\neq{0}\text{ and }{\rho}(z,w)<0}\right\}\neq\emptyset\text{.}
\end{align*}
The $\Omega$ we just defined is the set appearing in Theorem \ref{oberwurmmaintheorem}, so we have to show that $\Omega$ has the desired properties. We start by collecting some basic properties of $\Omega$ in a lemma, whose elementary proof will be omitted:

\theoremstyle{plain}
\newtheorem{wormpaperbasicprop}[propo]{Lemma}
\begin{wormpaperbasicprop}
\label{wormpaperbasicprop}
The set $\Omega$ is a bounded, connected open subset of $\mathbb{C}^2$ with boundary of class $\mathcal{C}^{1,1}$. Furthermore, the boundary of $\Omega$ (as a subset of $\mathbb{C}^2$) is precisely the set of all points $(z,w)\in\mathbb{C}^2$ satisfying $z\neq{0}$ and $\rho{(z,w)}=0$.
\end{wormpaperbasicprop}

\theoremstyle{remark}
\newtheorem*{wormpaperc11definition}{Remark}
\begin{wormpaperc11definition}
In this paper, we work with the following notion of $\mathcal{C}^{1,1}$-boundary: an open set $\emptyset\subsetneqq{U}\subsetneqq\mathbb{R}^k$ is said to have boundary of class $\mathcal{C}^{1,1}$, if for every boundary point $p$ of $U$ there exist an open neighborhood $V$ of $p$ in $\mathbb{R}^k$ and a function $r\colon{V}\to\mathbb{R}$ of class $\mathcal{C}^{1,1}$, such that $\nabla{r}$ vanishes nowhere on $V$ and $U\cap{V}=\{x\in{V}\colon{r(x)<0}\}$.
\end{wormpaperc11definition}

\theoremstyle{definition}
\newtheorem{setofradiusnommainwormpaper}[propo]{Notation}
\begin{setofradiusnommainwormpaper}
\label{setofradiusnommainwormpaper}
Let $M$ be a subset of $\mathbb{C}^n$ and let $r>0$. Then we define:
\begin{align*}
M(r):=\{z\in{\mathbb{C}^n}\colon{}\exists{x\in{M}}\text{ s.t.\ }||x-z||<r\}\text{.}
\end{align*}
$M(r)$ obviously is an open subset of $\mathbb{C}^n$.
\end{setofradiusnommainwormpaper}

We end this section with an {\emph{informal}} explanation of the intuition behind our constructions:

A classical worm domain admits a Stein neighborhood basis if the duration of the rotation at maximal radius is less than $\pi$. If the duration is exactly $\pi$ this fails to be true, as can be seen by refining the classical argument by K.\ Diederich and J. E. Forn{\ae}ss \cite{MR0430315}. In the case of the domain $\Omega$ defined above, we prevent this argument from working by drastically increasing the speed of the round-off, which leads to the boundary regularity dropping to $\mathcal{C}^{1,1}$. Using the fact that the function $g$ vanishes to infinite order in $0\in\mathbb{R}$, one can apply the Kontinuit{\" a}tssatz for annuli to open pseudoconvex neighborhoods of the closure of $\Omega$ to show that $\overline{\Omega}$ is not $s$-H-convex for any $s\geq{1}$. The details will be given in Section \ref{wormpapershconvsection}.

It is easy to construct a neighborhood basis for $\overline{\Omega}$ ({\emph{not}} a Stein one) by taking appropriate worm domains and increasing the radii of the rotating discs without changing the centers. This increase of the radii of course destroys pseudoconvexity. We counteract this by ``chopping off'' the ``bad part'', which is done by intersecting with a domain of half planes rotating around $0$ in the $w$-plane. This, however, leads to these sets not being neighborhoods anymore, as can be seen by considering $0$ in the $w$-plane. We finally resolve this issue by moving the center of the rotation from $0$ slightly in the direction of $-i$ and slightly slowing down the rotation (symmetrically around the angle ${\pi}/2$), which intuitively speaking amounts to introducing a small tilt. In the $w$-plane, $-i$ represents the ``out direction'' of $\Omega$, which exists because the duration of the rotation at maximal radius does not exceed $\pi$. Since $g$ is positive on $\mathbb{R}_{>0}$, one actually leaves the {\emph{closure}} of $\Omega$, when going from $0$ slightly in the direction of $-i$ in the $w$-plane, which is of course crucial for our construction to work. Since the purpose of the domain of rotating half planes is to help with the pseudoconvexity of the neighborhoods we are constructing, we have to apply these changes to both of the domains we are intersecting. The details will be given in Section \ref{wormpapersteinnhbdsection}.  

\section{Regarding $s$-H-Convexity}\label{wormpapershconvsection}
For this section we fix an ${\epsilon}_{0}>0$, such that $\sqrt{\mathcal{S}({\pi}+{\alpha})}+{\epsilon}_0<1$ and ${\epsilon}_0<g({\alpha})$. Given $0<{\epsilon}<{\epsilon}_0$, we define a map $H_{\epsilon}\colon{[{\pi},{\pi}+{\alpha}]}\to\mathbb{R}$ by
\begin{align*}
\phi\mapsto & \phantom{=}\sqrt{\mathcal{S}({\phi})}-\cos{({\phi}-{\pi})}+\frac{\epsilon}{2}\\
& =\frac{\epsilon}{2}-g({\phi}-{\pi})\text{.}
\end{align*}
By choice of ${\epsilon}_0$, we can apply the intermediate value theorem to find a zero $x_{\epsilon}\in{({\pi},{\pi}+{\alpha})}$ of $H_{\epsilon}$ for every $\epsilon\in{(0,{\epsilon}_0)}$, which is uniquely determined, since $H_{\epsilon}$ is strictly decreasing. By direct computation we get
\begin{align*}
{x_\epsilon}={\pi}+\frac{1}{-\ln{({\epsilon}/2)}}\text{ for all }\epsilon\in{(0,{\epsilon_0})}\text{.}
\end{align*}

{\emph{Roughly speaking}}, given some $0<{\epsilon}<{\epsilon}_0$ and an open pseudoconvex set containing ${\Omega}({\epsilon})$, we need to identify a point contained in said pseudoconvex set that is ``far away'' from $\Omega$ relative to $\epsilon$. By inspecting the explicit expression for $x_\epsilon$, one sees that ${x_\epsilon}-\pi$ is much larger than $\epsilon$ for small enough $0<\epsilon\ll{\epsilon_0}$. With this in mind, we will identify a point contained in any open pseudoconvex set containing $\Omega{({\epsilon})}$, whose distance to $\Omega$ is comparable to ${x_\epsilon}-\pi$. We accomplish this by applying the Kontinuit{\" a}tssatz for annuli.

The following lemma is the first step of the announced Kontinuit{\" a}tssatz argument. It deals with the boundaries of the annuli and the ``bottom annulus'':

\theoremstyle{plain}
\newtheorem{wormpaperannubd}[propo]{Lemma}
\begin{wormpaperannubd}
\label{wormpaperannubd}
Given $\epsilon\in{(0,{\epsilon}_0)}$, we have:
\begin{enumerate}
\item\label{wormpaperannubdprei}{For all $\phi$, ${\pi}\leq\phi\leq{x_\epsilon}$, the following set is contained in $\Omega{({\epsilon})}$:
\begin{align*}
\Big\{(z,w)\in\mathbb{C}^2\colon{|z|^2\in\{\exp{({\phi})},\exp{({\pi}-{\phi})}\}\text{ and }w=i\cdot\sin{({\phi})}}\Big\}\text{.}
\end{align*}}
\item\label{wormpaperannubdprzw}{The following set is contained in the boundary of ${\Omega}$ and hence in ${\Omega}({\epsilon})$:
\begin{align*}
\Big\{(z,w)\in\mathbb{C}^2\colon{\exp{(0)}\leq{}|z|^2\leq{\exp{({\pi})}}\text{ and }w=0}\Big\}\text{.}
\end{align*}
}
\end{enumerate}
\end{wormpaperannubd}

\begin{proof}
Property \ref{wormpaperannubdprzw} is clear, so we only need to prove Property \ref{wormpaperannubdprei}. Let $\epsilon\in{(0,{\epsilon}_0)}$, let ${\pi}\leq\phi\leq{x_\epsilon}$ and consider a point $(z,w)=(z,i\cdot{\sin{({\phi})}})$ contained in the set from the statement of Property \ref{wormpaperannubdprei}. We restrict ourselves to the case $|z|^2={\exp{({\phi})}}$, since the other case can be handled analogously.\\
But then, owing to the choices we made, $(z,\widetilde{w})$ is contained in $\Omega$, whenever $\widetilde{w}$ is contained in the open disc in $\mathbb{C}$ centered at $\exp{(i\cdot{\phi})}$ with radius $\sqrt{\mathcal{S}({\phi})}>0$. So it suffices to prove that $|w-\exp{({i\cdot{\phi}})}|$ is less than $\sqrt{\mathcal{S}({\phi})}+{\epsilon}$.

Making use of the choices made above (in particular that $\pi{<}{x_\epsilon}<{\pi}+\alpha{<{{\pi}+\beta}}{<\pi}+{\pi}/2$ and $H_\epsilon\geq{0}$ on $[{\pi},x_{\epsilon}]$), we compute:
\begin{align*}
|w-\exp{({i\cdot{\phi}})}| & =|i\cdot{\sin{({\phi})}}-\exp{({i\cdot{\phi}})}|\\
& =\sqrt{(\cos{({\phi})})^2}\\
& ={\cos{({\phi}-{\pi})}}\\
& \leq{\cos{({\phi}-{\pi})}}+H_{\epsilon}({\phi})\\
& =\sqrt{\mathcal{S}({\phi})}+{\epsilon}/2\\
& <\sqrt{\mathcal{S}({\phi})}+{\epsilon}\text{,}
\end{align*}
as desired.
\end{proof}

Armed with Lemma \ref{wormpaperannubd}, we now finish the Kontinuit{\" a}tssatz argument:

\theoremstyle{plain}
\newtheorem{wormpaperannupush}[propo]{Lemma}
\begin{wormpaperannupush}
\label{wormpaperannupush}
Let $\epsilon\in{(0,{\epsilon}_0)}$ and let $D\subseteq\mathbb{C}^2$ be an open pseudoconvex set containing ${\Omega}({\epsilon})$.\\
Then, for every $\phi$, ${\pi}\leq\phi\leq{x_\epsilon}$, the following set is contained in $D$:
\begin{align*}
{F_\phi}:=\Big\{(z,w)\in\mathbb{C}^2\colon{\exp{({\pi}-{\phi})}\leq{}|z|^2\leq{\exp{({\phi})}}\text{ and }w=i\cdot\sin{({\phi})}}\Big\}\text{.}
\end{align*}
\end{wormpaperannupush}

\begin{proof}
This follows from Lemma \ref{wormpaperannubd} via the Kontinuit{\" a}tssatz for annuli.
\end{proof}

In view of Lemma \ref{wormpaperannupush}, we need to identify a point contained in $F_{x_\epsilon}$ that is ``far away'' from $\Omega$. The obvious choice is the following:

For all $\epsilon\in{(0,{\epsilon}_0)}$ we define
\begin{align*}
p_{\epsilon}:=\left(\exp{\left(\frac{\pi}{4}\right)},i\cdot\sin{(x_{\epsilon})}\right)\in{F_{x_\epsilon}}\text{.}
\end{align*}

The following lemma shows that $p_\epsilon$ is indeed ``far away'' from $\Omega$:

\theoremstyle{plain}
\newtheorem{wormpaperlipdist}[propo]{Lemma}
\begin{wormpaperlipdist}
\label{wormpaperlipdist}
There exist constants $L>0$ and $\delta{>0}$, such that for all $\epsilon\in{(0,{\epsilon}_0)}$ we have
\begin{align*}
d({p_\epsilon},{\Omega})\geq\min\left\{{\delta},\frac{{x_\epsilon}-\pi}{L}\right\}\text{,}
\end{align*}
where $d(\cdot{},{\Omega})$ denotes the Euclidean distance of a point in $\mathbb{C}^2$ to $\Omega$.
\end{wormpaperlipdist}

\begin{proof}
Owing to Lemma \ref{wormpaperbasicprop}, we find a ${\delta}>0$, such that $\overline{{\Omega}({\delta})}$, the closure of $\Omega{({\delta})}$ in $\mathbb{C}^2$, is a compact subset of $(\mathbb{C}\setminus\{0\})\times\mathbb{C}$. So, since $\rho$ is of class $\mathcal{C}^1$ on $(\mathbb{C}\setminus\{0\})\times\mathbb{C}$ (see Section \ref{wormpaperprelimsection}), there exists an $L>0$, such that $\rho$ is Lipschitz continuous with Lipschitz constant $L$ on $\overline{{\Omega}({\delta})}$. That immediately gives the estimate
\begin{align*}
d(p,{\Omega})\geq\min\left\{{\delta},\frac{1}{L}\cdot\rho{(p)}\right\}\text{ for all }p\in{(\mathbb{C}\setminus\{0\})\times\mathbb{C}}\text{.}
\end{align*}
Hence, given $\epsilon\in{(0,{\epsilon}_0)}$, we only need to show that ${\rho}({p_\epsilon})\geq{x_\epsilon}-\pi$. Using that ${x_\epsilon}\in{({\pi},{\pi}+{\alpha})}$ and using the defining properties of $\alpha$, we compute:
\begin{align*}
{\rho}(p_{\epsilon}) & =\left|{i\cdot\sin{(x_{\epsilon})}-\exp{(i\cdot{\pi}/2)}}\right|^2-\mathcal{S}({\pi}/2)\\
& =\left|{i\cdot\sin{(x_{\epsilon})}-i}\right|^2-1\\
& \geq{}-2\sin{({x_\epsilon})}\\
& =2\sin{({x_\epsilon}-{\pi})}\\
& \geq{x_\epsilon}-\pi\text{,}
\end{align*}
as desired.
\end{proof}

We now combine all the previously developed ingredients to achieve the goal of this section:

\theoremstyle{plain}
\newtheorem{wormpapershnotpropos}[propo]{Proposition}
\begin{wormpapershnotpropos}
\label{wormpapershnotpropos}
$\overline{\Omega}$ is not $s$-H-convex for any real number $s\geq{1}$.
\end{wormpapershnotpropos}

\begin{proof}
First note that $\overline{\Omega}$ is indeed compact. Assume for the sake of a contradiction that $\overline{\Omega}$ is $s$-H-convex for some $s\geq{1}$. So there exist a constant $0<C\leq{1}$ and a family $(D_{\epsilon})_{0<\epsilon\leq{1}}$ of open pseudoconvex subsets of $\mathbb{C}^2$, such that
\begin{align*}
\Omega{(C\cdot{\epsilon}^s)}\subseteq{D_\epsilon}\subseteq\Omega{({\epsilon})}\text{ for all }0<\epsilon\leq{1}\text{,}
\end{align*}
i.e.\ we have
\begin{align*}
\Omega{({\epsilon})}\subseteq{D_{({\epsilon}/C)^{1/s}}}\subseteq\Omega{(({\epsilon}/C)^{1/s})}\text{ for all }0<\epsilon\leq{C}\text{.}
\end{align*}
For all $0<{\epsilon}<\min\{{\epsilon_0},C\}$ we then get from Lemma \ref{wormpaperannupush} that
\begin{align*}
p_{\epsilon}\in{F_{x_\epsilon}}\subseteq{D_{({\epsilon}/C)^{1/s}}}\subseteq\Omega{(({\epsilon}/C)^{1/s})}\text{,}
\end{align*}
which, using Lemma \ref{wormpaperlipdist}, directly implies the estimate
\begin{align*}
\min\left\{{\delta},\frac{{x_\epsilon}-\pi}{L}\right\}\leq{d({p_\epsilon},{\Omega})}<{\left({\frac{\epsilon}{C}}\right)}^{1/s}\text{ for all }0<{\epsilon}<\min\{{\epsilon_0},C\}\text{.}
\end{align*}
So, since $\delta$, $L$ and $C$ are positive constants, we find a constant $K>0$ and an $0<\widehat{\epsilon}\ll\min\{{\epsilon_0},C\}$, such that
\begin{align*}
({x_\epsilon}-{\pi})^s<K\epsilon\text{ for all }0<{\epsilon}<\widehat{\epsilon}\text{.}
\end{align*}
Using that
\begin{align*}
0=H_{\epsilon}(x_{\epsilon})={\epsilon}/2-g({x_\epsilon}-{\pi})={\epsilon}/2-{\exp}(-1/({x_\epsilon}-{\pi}))\text{,}
\end{align*}
for all $0<{\epsilon}<{\epsilon}_0$, we get that
\begin{align*}
{\left({\frac{1}{-\ln{({\epsilon}/2)}}}\right)}^s<K{\epsilon}\text{ for all }0<{\epsilon}<\widehat{\epsilon}\text{,}
\end{align*}
and we arrive at the desired contradiction.
\end{proof}

\section{Existence of a Stein Neighborhood Basis}\label{wormpapersteinnhbdsection}

In this section we construct a Stein neighborhood basis for $\overline{\Omega}$. We {\emph{fix}} an ${\epsilon}>0$ for the remainder of this section. It suffices to find an open pseudoconvex subset $D$ of $\mathbb{C}^2$ satisfying $\overline{\Omega}\subseteq{D}\subseteq{\Omega}({\epsilon})$.

We start by defining the domains of ``half planes rotating in the $w$-plane'' announced in Section \ref{wormpaperprelimsection}.

\theoremstyle{definition}
\newtheorem{wormpaperrotatinghp}[propo]{Definition}
\begin{wormpaperrotatinghp}
\label{wormpaperrotatinghp}
For every $\delta\in{(0,1)}$ and every $t\in{[0,1)}$ we define $H_{t}^{({\delta})}$ to be the subset of $\mathbb{C}^2$ consisting of all points $(z,w)$ satisfying $z\neq{0}$ and
\begin{align*}
t<\operatorname{Re}\Bigg( & {\left({w+i\cdot\sin{\left(\frac{\delta\pi}{2(1-{\delta})}\right)}}\right)}\\
& \cdot\exp\left({-i\cdot\left({\frac{\delta\pi}{2}+(1-{\delta})\ln{({{|z|}^2})}}\right)}\right)\Bigg)\text{.}
\end{align*}
Furthermore we will denote the set $H_{0}^{({\delta})}$ simply as $H^{({\delta})}$.
\end{wormpaperrotatinghp}

The expression $\sin{({\delta\pi}/(2(1-{\delta})))}$ measures by how much the center of the rotation is moved in the direction of $-i$. In the exponential-term, $\delta$ measures how much the rotation is slowed down symmetrically around the angle ${\pi}/2$. The expression $\sin{({\delta\pi}/(2(1-{\delta})))}$ was chosen specifically to ensure that an appropriate version of Lemma \ref{wormpapercrucialestimate} (see below) holds true.

Before we can define the domains of ``discs rotating in the $w$-plane'', we need to approximate $\mathcal{S}$ from above by smooth concave functions:

\theoremstyle{plain}
\newtheorem{wormpaperessetadef}[propo]{}
\begin{wormpaperessetadef}
\label{wormpaperessetadef}
There exists an ${\eta_0}$, $0<{\eta_0}\ll{1/2}$, such that for all $\eta\in{(0,{\eta_0})}$ there exist a $\mathcal{C}^\infty$-function $\mathcal{S}_{\eta}\colon\mathbb{R}\to\mathbb{R}$, a ${\beta_\eta}>\beta$ and an $x_\eta{>\pi}$ (not to be confused with the $x_\epsilon$ appearing in Section \ref{wormpapershconvsection}) with the following properties:
\begin{enumerate}
\item{$\mathcal{S}_\eta$ is concave on $\mathbb{R}$ and satisfies $\mathcal{S}_{\eta}(x+{\pi}/2)=\mathcal{S}_{\eta}(-x+{\pi}/2)$ for all $x\in\mathbb{R}$,}
\item{$\mathcal{S}_\eta$ is $\leq{1+\eta}$ on $\mathbb{R}$ and $\equiv{1+\eta}$ on a neighborhood of $[0,{\pi}]$ in $\mathbb{R}$,}
\item{$\mathcal{S}+{\eta}/2\leq\mathcal{S}_{\eta}\leq\mathcal{S}+{3\eta}/2$ on $\mathbb{R}$,}
\item{$S_{\eta}({\pi}+{\beta_\eta})=0$ and $S_{\eta}'({\pi}+{\beta_\eta})\neq{0}$,}
\item{$S_\eta$ is $>0$ on $({-\beta_\eta},\pi{+\beta_\eta})$ and $<0$ on $\mathbb{R}\setminus{[{-\beta_\eta},\pi{+\beta_\eta}]}$,}
\item{$x_\eta\in{}({\pi},\pi+{\beta_\eta})$ and $\mathcal{S}_{\eta}(x_{\eta})=1$; furthermore, $\mathcal{S}_\eta$ is $>1$ on $({\pi},x_{\eta})$ and $<1$ on $(x_{\eta},{\infty})$,}
\item{we have $-\mathcal{S}_{\eta}''({\phi})\geq{100}|\mathcal{S}_{\eta}'({\phi})|$, whenever ${\pi}/2\leq\phi\leq{x_\eta}$.}
\end{enumerate}
\end{wormpaperessetadef}

\begin{proof}
For all $\gamma{>0}$ we fix a $\mathcal{C}^\infty$-function $\Phi_{\gamma}\colon\mathbb{R}\to\mathbb{R}$, such that
\begin{itemize}
\item{$\Phi_{\gamma}'\geq{0}$ on $\mathbb{R}$,}
\item{$\Phi_{\gamma}\equiv{0}$ on $(-{\infty},{\pi}+{\gamma}/4]$ and $\Phi_{\gamma}\equiv{1}$ on $[{\pi}+3{\gamma}/4,{\infty})$.}
\end{itemize}
One now readily checks that, if $0<{\eta_0}\ll{1/2}$ is chosen {\emph{small enough}}, then, for all $\eta\in{(0,{\eta_0})}$, one can pick a {\emph{small}} $0<{\gamma}({\eta})\ll\alpha$, such that the function $S_{\eta}\colon\mathbb{R}\to\mathbb{R}$, given by
\begin{align*}
x\mapsto{1+\eta}+\int_{\frac{\pi}{2}}^{\frac{\pi}{2}+\left|{x-\frac{\pi}{2}}\right|}{\mathcal{S}'(t)\cdot\Phi_{{\gamma}({\eta})}(t)\text{ }dt}
\end{align*}
and the implicitly defined ${\beta}_\eta$ and $x_\eta$ have all the desired properties.
\end{proof}

We now define the domains of ``discs rotating in the $w$-plane'', also announced in Section \ref{wormpaperprelimsection}. It is important to note that these domains are {\emph{not}} pseudoconvex:

\theoremstyle{definition}
\newtheorem{wormpaperrotatingdisc}[propo]{Definition}
\begin{wormpaperrotatingdisc}
\label{wormpaperrotatingdisc}
Adopt the notation from \ref{wormpaperessetadef}. Then, for all $\delta\in{(0,1)}$ and for all $\eta\in{(0,{\eta_0})}$, we define a map ${\rho}_{{\delta},{\eta}}\colon{(\mathbb{C}\setminus\{0\})}\times\mathbb{C}\to\mathbb{R}$ by
\begin{align*}
(z,w)\mapsto & {\left|{w+i\cdot\sin{\left(\frac{\delta\pi}{2(1-{\delta})}\right)}-\exp\left({i\cdot\left({\frac{\delta\pi}{2}+(1-{\delta})\ln{({{|z|}^2})}}\right)}\right)}\right|}^2\\
& -\mathcal{S}_{\eta}\left({{{\frac{\delta\pi}{2}+(1-{\delta})\ln{({{|z|}^2})}}}}\right)\text{,}
\end{align*}
and we define ${D}^{({\delta},{\eta})}$ to be the subset of $\mathbb{C}^2$ consisting of all points $(z,w)$ satisfying $z\neq{0}$ and ${\rho}_{{\delta},\eta}(z,w)<0$.
\end{wormpaperrotatingdisc}

It should be noted that ${D}^{({\delta},{\eta})}$ is essentially defined the same way as $\Omega$ (resp.\ a classical worm domain), apart from the fact that $\mathcal{S}$ is replaced by $\mathcal{S}_\eta$ and that the position of the center and the speed of the rotation have been adjusted slightly (in the same way as above).

We now show that $\overline{\Omega}\subseteq{D}^{({\delta},{\eta})}\subseteq{\Omega}({\epsilon})$ for suitable choices of $\delta$ and $\eta$. Since ${D}^{({\delta},{\eta})}$ is {\emph{not}} pseudoconvex, however, some additional considerations are needed in order to achieve the goal stated in the beginning of this section.

\theoremstyle{plain}
\newtheorem{wormpaperrdconclos}[propo]{Lemma}
\begin{wormpaperrdconclos}
\label{wormpaperrdconclos}
There exists an ${\eta}_{1}({\epsilon})\in{(0,{\eta_0})}$, such that for each $\eta\in{(0,{\eta_1}({\epsilon}))}$ there exists a ${d_2}({\epsilon},{\eta})\in{(0,1/2)}$ with the property that
\begin{align*}
\overline{\Omega}\subseteq{D}^{({\delta},{\eta})}\subseteq{\Omega}({\epsilon})\text{,}
\end{align*}
whenever $0<{\delta}<{d_2}({\epsilon},{\eta})$.
\end{wormpaperrdconclos}

\begin{proof}
This follows from a straightforward calculation using the properties in \ref{wormpaperessetadef}.
\end{proof}

As explained in Section \ref{wormpaperprelimsection}, we want to intersect the domains of ``discs rotating in the $w$-plane'' with suitable domains of ``half planes rotating in the $w$-plane'', with the aim of obtaining a pseudoconvex neighborhood of $\overline{\Omega}$. So we of course need the domains of ``half planes rotating in the $w$-plane'' to contain the closure of $\Omega$.

In order to establish this, we need the crucial estimate provided by Lemma \ref{wormpapercrucialestimate} below. If the function $g$ was replaced by the $0$-function in a small neighborhood of $0\in\mathbb{R}$, then $\overline{\Omega}$ could not possibly have a Stein neighborhood basis, as the Kontinuit{\" a}tssatz for annuli shows. Hence our construction {\emph{has to}} make use of the fact that $g>0$ on an interval of the form $(0,{\mu})$ for some small $0<{\mu}\ll{1}$. We make use of that fact only once in the entire construction of the Stein neighborhood basis for $\overline{\Omega}$, namely in the proof of Lemma \ref{wormpapercrucialestimate}, the discovery of which was one of the main obstacles in our construction. In fact, the seemingly arbitrary expression $\sin{({\delta\pi}/(2(1-{\delta})))}$ featuring in Definition \ref{wormpaperrotatinghp} was chosen specifically with this lemma in mind.

\theoremstyle{plain}
\newtheorem{wormpapercrucialestimate}[propo]{Lemma}
\begin{wormpapercrucialestimate}
\label{wormpapercrucialestimate}
There exists a $0<{d_1}<1$, such that we have the following estimate for all ${\delta},{\psi}\in\mathbb{R}$ with $0<{\delta}<d_1$ and $-\beta\leq\psi\leq{\pi}+\beta$:
\begin{align*}
0 & <\cos{\left({\delta\left({\frac{\pi}{2}-\psi}\right)}\right)}-\sqrt{\mathcal{S}({\psi})}\\
& \phantom{<}+\sin{\left({{\psi}+\delta\left({\frac{\pi}{2}-\psi}\right)}\right)}\cdot\sin{\left({\frac{\delta\pi}{2({1-\delta})}}\right)}\text{.}
\end{align*}
\end{wormpapercrucialestimate}

The proof of Lemma \ref{wormpapercrucialestimate} can be found in Section \ref{wormpapertechnlemmasection}. Using this lemma, we can now show that the domains of ``half planes rotating in the $w$-plane'' contain the closure of ${\Omega}$.

\theoremstyle{plain}
\newtheorem{wormpaperrhpconclos}[propo]{Lemma}
\begin{wormpaperrhpconclos}
\label{wormpaperrhpconclos}
Let $d_1\in{(0,1)}$ be as in Lemma \ref{wormpapercrucialestimate}. Then, given $\delta\in{(0,{d_1})}$, there exists a $t_{\delta}\in{(0,1)}$, such that
\begin{align*}
\overline{\Omega}\subseteq{H_{t}^{({\delta})}}
\end{align*}
for all $0<t<t_{\delta}$.
\end{wormpaperrhpconclos}

\begin{proof}
Let $\delta\in{(0,{d_1})}$. Owing to the compactness of $\overline{\Omega}$, it suffices to show that $\overline{\Omega}\subseteq{H_{0}^{({\delta})}}$. To this end let $(z,w)\in\overline{\Omega}$. Lemma \ref{wormpaperbasicprop} shows that $z\neq{0}$ and $\rho{(z,w)}\leq{0}$. In particular, this implies that $\psi{:=\ln}({|z|}^2)\in{[{-\beta},{{\pi}+\beta}]}$ and $|w-\exp{(i{\psi})}|\leq\sqrt{\mathcal{S}({\psi})}$.

Hence, using that $\operatorname{Re}({\tau})\geq{-|{\tau}|}$ for all $\tau\in\mathbb{C}$ and writing $w=\exp{(i{\psi})}+(w-\exp{(i{\psi})})$, we get
\begin{align*}
& \operatorname{Re}\Bigg({\left({w+i\cdot\sin{\left(\frac{\delta\pi}{2(1-{\delta})}\right)}}\right)}\\
& \phantom{\operatorname{Re}\Bigg(){}}\cdot\exp\left({-i\cdot\left({\frac{\delta\pi}{2}+(1-{\delta})\ln{({{|z|}^2})}}\right)}\right)\Bigg)\\
\geq & \operatorname{Re}\Bigg({\left({\exp{(i{\psi})}+i\cdot\sin{\left(\frac{\delta\pi}{2(1-{\delta})}\right)}}\right)}\\
& \phantom{\operatorname{Re}\Bigg(){}}\cdot\exp\left({-i\cdot\left({\frac{\delta\pi}{2}+(1-{\delta}){\psi}}\right)}\right)\Bigg)\\
& -|w-\exp{(i{\psi})}|\\
\geq & \cos{\left({\delta\left({\frac{\pi}{2}-\psi}\right)}\right)}-\sqrt{\mathcal{S}({\psi})}\\
& +\sin{\left({{\psi}+\delta\left({\frac{\pi}{2}-\psi}\right)}\right)}\cdot\sin{\left({\frac{\delta\pi}{2({1-\delta})}}\right)}\text{,}
\end{align*}
which is $>0$ by Lemma \ref{wormpapercrucialestimate}. This shows that $(z,w)\in{H_{0}^{({\delta})}}$, as desired.
\end{proof}

We are now ready to define the Stein neighborhood announced in the beginning of this section. Adopting the notation from Lemmas \ref{wormpaperrdconclos} and \ref{wormpaperrhpconclos}, we {\emph{fix}} an $\eta\in{(0,{\eta_1}({\epsilon}))}$, a $\delta{>0}$ with $\delta{<\min}\{{d_1},{d_2}{({\epsilon},{\eta})}\}$ and a $t\in{(0,{t_{\delta}})}$ for the remainder of this section. With these fixed choices we now define
\begin{align*}
D:={D^{({\delta},{\eta})}}\cap{H_{t}^{({\delta})}}\text{.}
\end{align*}
It is obvious that $D$ is an open subset of $\mathbb{C}^2$. Furthermore, we have
\begin{align*}
\overline{\Omega}\subseteq{D}\subseteq\Omega{({\epsilon})}
\end{align*}
by Lemmas \ref{wormpaperrdconclos} and \ref{wormpaperrhpconclos}. Hence, we only have to show that $D$ is pseudoconvex.

Pseudoconvexity is a local property of the boundary and we have
\begin{align*}
bD\subseteq & \phantom{\cup}\bigg({\left({b{H_{t}^{({\delta})}}}\right)}\cap{\left({b{D^{({\delta},{\eta})}}}\right)}\bigg)\\
& \cup\bigg({\left({b{H_{t}^{({\delta})}}}\right)}\cap{\left({{D^{({\delta},{\eta})}}}\right)}\bigg)\\
& \cup\bigg({\left({{H_{t}^{({\delta})}}}\right)}\cap{\left({b{D^{({\delta},{\eta})}}}\right)}\bigg)\text{.}
\end{align*}
So, since the boundary $bD$ of $D$ is contained in $H_{0}^{({\delta})}\subseteq{(\mathbb{C}\setminus\{0\})}\times\mathbb{C}$, pseudoconvexity of $D$ follows from the following two lemmas, the proofs of which can be found in Section \ref{wormpapertechnlemmasection}.

\theoremstyle{plain}
\newtheorem{wormpaperd1locpscx}[propo]{Lemma}
\begin{wormpaperd1locpscx}
\label{wormpaperd1locpscx}
Let $({z_0},{w_0})\in{b{H_{t}^{({\delta})}}}$ and assume that $({z_0},{w_0})\in{{H_{0}^{({\delta})}}}$. Then there exists an open neighborhood $V$ of $({z_0},{w_0})$ in $\mathbb{C}^2$, such that $V\cap{{H_{t}^{({\delta})}}}$ is pseudoconvex.
\end{wormpaperd1locpscx}

\theoremstyle{plain}
\newtheorem{wormpaperd2locpscx}[propo]{Lemma}
\begin{wormpaperd2locpscx}
\label{wormpaperd2locpscx}
Let $({z_0},{w_0})\in{b{{D^{({\delta},{\eta})}}}}$ and assume that $({z_0},{w_0})\in{{H_{0}^{({\delta})}}}$. Then there exists an open neighborhood $V$ of $({z_0},{w_0})$ in $\mathbb{C}^2$, such that $V\cap{{D^{({\delta},{\eta})}}}$ is pseudoconvex.
\end{wormpaperd2locpscx}

Lemma \ref{wormpaperd1locpscx} deals with the pseudoconvexity of our chosen domain of ``half planes rotating in the $w$-plane'' at certain boundary points. Lemma \ref{wormpaperd2locpscx} says, roughly speaking, that our chosen domain of ``discs rotating in the $w$-plane'' is pseudoconvex at the ``good'' boundary points, which are precisely those contained in $H_{0}^{({\delta})}$.

As mentioned previously, pseudoconvexity of $D$ follows from Lemmas \ref{wormpaperd1locpscx} and \ref{wormpaperd2locpscx}, the proofs of which can be found in Section \ref{wormpapertechnlemmasection}; so we have shown that $D$ is pseudoconvex. Hence $\overline{\Omega}$ has a Stein neighborhood basis. Together with Proposition \ref{wormpapershnotpropos}, this provides a proof for Theorem \ref{oberwurmmaintheorem}.

\section{Remaining Proofs}\label{wormpapertechnlemmasection}

In this section we provide the proofs which remain from Section \ref{wormpapersteinnhbdsection}. We start by proving the crucial estimate, Lemma \ref{wormpapercrucialestimate}.

\begin{proof}[Proof of Lemma \ref{wormpapercrucialestimate}]
Note first that the expression in the claimed inequality is indeed welldefined, since $\mathcal{S}\geq{0}$ on $[-{\beta},{\pi}+{\beta}]$. Owing to the symmetry of $\mathcal{S}$, see \ref{wormpapersissymmy} in Section \ref{wormpaperprelimsection}, we can restrict ourselves to considering the case where $\psi\in{[{-\beta},{\pi}/2]}$. Noting that ${[{-\beta},{\pi}/2]}={[{-\beta},-{\alpha}/2]}\cup{[0,{\pi}/2]}\cup{(-{\alpha}/2,0)}$, we will consider the three intervals on the right hand side separately. Pick some $d_1\in{(0,1/4)}$. By a slight {\emph{abuse of notation}}, we will shrink $d_1$ a finite amount of times over the course of the proof, until it has the desired property.

First consider the interval $[-{\beta},-{\alpha}/2]$. Using Properties \ref{wormpapersissymmy}, \ref{wormpaperprelimprop4} and \ref{wormpaperprelimprop5} in Section \ref{wormpaperprelimsection}), we find a ${\Delta}>0$, such that $\sqrt{\mathcal{S}}<1-\Delta$ on $[-{\beta},-{\alpha}/2]$. By making ${d_1}\in{(0,1)}$ smaller if necessary, we have
\begin{align*}
\left|{\sin{\left({\frac{\delta\pi}{2({1-\delta})}}\right)}}\right|{<\frac{\Delta}{4}}\text{ and }\cos{\left({\delta\left({\frac{\pi}{2}-\psi}\right)}\right)}>1-\frac{\Delta}{4}
\end{align*}
for all $\delta\in{(0,{d_1})}$ and $\psi\in{[-{\beta},-{\alpha}/2]}$. The claimed inequality is then clear in this case.

Next consider the interval $[0,{\pi}/2]$. On this interval we have $\sqrt{\mathcal{S}}\equiv{1}$. By making $d_1\in{(0,1)}$ smaller if necessary, we have the following for all $\delta\in{(0,{d_1})}$:
\begin{align*}
0<\delta\cdot\frac{\pi}{2}<\frac{\delta\pi}{2(1-{\delta})}<\frac{\pi}{2}\text{.}
\end{align*}
We compute, for $\delta\in{(0,d_1)}$ and $\psi\in{[0,{\pi}/2]}$:
\begin{align*}
& \cos{\left({\delta\left({\frac{\pi}{2}-\psi}\right)}\right)}-\sqrt{\mathcal{S}({\psi})}\\
& +\sin{\left({{\psi}+\delta\left({\frac{\pi}{2}-\psi}\right)}\right)}\cdot\sin{\left({\frac{\delta\pi}{2({1-\delta})}}\right)}\\
\geq & \cos\left(\delta\cdot\frac{\pi}{2}\right){-1}+\sin\left(\delta\cdot\frac{\pi}{2}\right)\cdot\sin{\left({\frac{\delta\pi}{2({1-\delta})}}\right)}\\
> & \cos\left(\delta\cdot\frac{\pi}{2}\right){-1}+{\left(\sin\left(\delta\cdot\frac{\pi}{2}\right)\right)}^2\\
= & \cos\left(\delta\cdot\frac{\pi}{2}\right)\cdot\left({1-{\cos\left(\delta\cdot\frac{\pi}{2}\right)}}\right)\text{,}
\end{align*}
which is larger than $0$, as desired.

Finally consider the interval $(-{\alpha}/2,0)$. By shrinking $d_1$ if necessary, we can assume that $\delta\pi{/}(2(1-{\delta}))\in{(0,{\alpha}/2)}$ for all $\delta\in{(0,d_1)}$. For ease of notation, we define a function $M\colon{(0,1)}\times\mathbb{R}\to\mathbb{R}$ by
\begin{align*}
(t,y)\mapsto & \phantom{+}\cos\left(\frac{t\pi}{2(1-t)}\right)\cdot\Big({\cos{(ty)}-\cos{(y)}}\Big)\\
& +\sin\left(\frac{t\pi}{2(1-t)}\right)\cdot\Big({\sin{(ty)}+\sin{((1-t)y)}-\sin{(y)}}\Big)\text{,}
\end{align*}
and a function $\phi\colon{(0,1)}\times\mathbb{R}\to\mathbb{R}$ by
\begin{align*}
(t,x)\mapsto{x+\frac{t\pi}{2(1-t)}}\text{.}
\end{align*}
Using Properties \ref{wormpapersissymmy} and \ref{wormpaperprelimprop7} in Section \ref{wormpaperprelimsection} as well as some elementary trigonometric identities, one readily checks that we have the following for all $\psi\in(-{\alpha}/2,0)$ and $\delta\in{(0,{d_1})}$:
\begin{align*}
g(-{\psi})+M({\delta},{\phi}({\delta},{\psi})) & =\cos{\left({\delta\left({\frac{\pi}{2}-\psi}\right)}\right)}-\sqrt{\mathcal{S}({\psi})}\\
& \phantom{<}+\sin{\left({{\psi}+\delta\left({\frac{\pi}{2}-\psi}\right)}\right)}\cdot\sin{\left({\frac{\delta\pi}{2({1-\delta})}}\right)}\text{.}
\end{align*}
So, since $g>0$ on $\mathbb{R}_{>0}$, it suffices to prove that $M({\delta},{\phi}({\delta},{\psi}))\geq{0}$ for all $\psi\in(-{\alpha}/2,0)$ and $\delta\in{(0,{d_1})}$.

First we consider the case where ${\phi}({\delta},{\psi})\geq{0}$ for some $\psi\in(-{\alpha}/2,0)$ and $\delta\in{(0,{d_1})}$. Since $\delta\pi{/}(2(1-{\delta}))\in{(0,{\pi}/2)}$ and $0<{\delta}<1$ and ${\phi}({\delta},{\psi})\in{[0,{\pi}/2)}$, it suffices to prove the following two inequalities for all $y\in{[0,{\pi}/2]}$ and $t\in{[0,1]}$:
\begin{align*}
{\cos{(ty)}-\cos{(y)}} & \geq{0}\text{,}\\
{\sin{(ty)}+\sin{((1-t)y)}-\sin{(y)}} & \geq{0}\text{.}
\end{align*}
The first inequality is trivial and the second inequality is obvious from the fact that $\sin{(0)}=0$ and $\sin$ is concave on $[0,{\pi}/2]$.

Finally, consider some $\psi\in(-{\alpha}/2,0)$ and $\delta\in{(0,{d_1})}$, for which $\phi{({\delta},{\psi})}{<0}$. For ease of notation, we simply write $\phi$ for $\phi{({\delta},{\psi})}$. Since $-{\alpha}/2<{\psi}<{\phi}<0$ and ${d_1}<1/4$ and by Property \ref{wormpaperprelimprop6} in Section \ref{wormpaperprelimsection}, we have the following estimates for an appropriate $\xi\in{({\phi},{\delta\phi})}$ (coming from the mean value theorem):
\begin{align*}
\cos{({\delta\phi})}-\cos{({\phi})} & =({\delta\phi}-{\phi})\cdot{(-\sin{({\xi})})}\\
& =\sin{(-{\xi})}\cdot{(1-{\delta})}\cdot{(-{\phi})}\\
& \geq\left|\sin{(-\delta{\phi})}\right|\cdot{(1-{\delta})}\cdot{|{\phi}|}\\
& \geq{(3/4)}\cdot{|{-\delta\phi}|}\cdot{(3/4)}\cdot{|{\phi}|}\\
& >{\delta}{|{\phi}|^2}/2\text{,}
\end{align*}
and
\begin{align*}
\sin{({\delta\phi})}+\sin{((1-{\delta}){\phi})}-\sin{({\phi})} & \geq{-\sin{({\phi})}}-|{\delta\phi}|-|(1-{\delta}){\phi}|\\
& =-(\sin{(\phi)}+|{\phi}|)\\
& =-(\sin{({\phi})}-{\phi})\\
& \geq{-}|\sin{({\phi})}-{\phi}|\\
& \geq{-}|{\phi}|^3\text{,}
\end{align*}
and finally, since $\delta\pi{/}(2(1-{\delta}))\in{(0,{\alpha}/2)}\subseteq{(0,{\pi}/2)}$ and ${\alpha}<1/({4\pi})$, we can conclude that
\begin{align*}
M({\delta},{\phi}({\delta},{\psi})) & =\cos{\left({\frac{\delta\pi}{2({1-\delta})}}\right)}\cdot\Bigg(\cos{({\delta\phi})}-\cos{({\phi})}\\
& \phantom{=}+\tan{\left({\frac{\delta\pi}{2({1-\delta})}}\right)}\cdot\Big({\sin{({\delta\phi})}+\sin{((1-{\delta}){\phi})}-\sin{({\phi})}}\Big)\Bigg)\\
& \geq\cos{\left({\frac{\delta\pi}{2({1-\delta})}}\right)}\cdot\Bigg({\frac{1}{2}\delta{|{\phi}|}^2-\tan{\left({\frac{\delta\pi}{2({1-\delta})}}\right)}\cdot{|{\phi}|}^3}\Bigg)\\
& \geq\cos{\left({\frac{\delta\pi}{2({1-\delta})}}\right)}\cdot\Bigg({\frac{1}{2}\delta{|{\phi}|}^2-2\cdot{\left|{\frac{\delta\pi}{2({1-\delta})}}\right|}\cdot{|{\phi}|}^3}\Bigg)\\
& ={\frac{1}{2}\delta{|{\phi}|}^2}\cdot\cos{\left({\frac{\delta\pi}{2({1-\delta})}}\right)}\cdot\Bigg({1-2\cdot{{\frac{\pi}{{1-\delta}}}}\cdot{|{\phi}|}}\Bigg)\\
& \geq{\frac{1}{2}\delta{|{\phi}|}^2}\cdot\cos{\left({\frac{\delta\pi}{2({1-\delta})}}\right)}\cdot\left({1-2\pi\cdot\frac{4}{3}\cdot\frac{1}{4\pi}}\right)\text{,}
\end{align*}
which is clearly $\geq{0}$, as desired.
\end{proof}

It remains to prove Lemmas \ref{wormpaperd1locpscx} and \ref{wormpaperd2locpscx}. Over the course of Section \ref{wormpapersteinnhbdsection} we fixed choices of $\epsilon$, $\delta$, $\eta$ and $t$. We of course work with those choices in the proofs of said lemmas. We start with the proof of Lemma \ref{wormpaperd1locpscx}:

\begin{proof}[Proof of Lemma \ref{wormpaperd1locpscx}]
We define a map $r\colon{({\mathbb{C}\setminus\{0\}})}\times\mathbb{C}\to\mathbb{R}$ by
\begin{align*}
(z,w)\mapsto{t}-\operatorname{Re}\Bigg( & {\left({w+i\cdot\sin{\left(\frac{\delta\pi}{2(1-{\delta})}\right)}}\right)}\\
& \cdot\exp\left({-i\cdot\left({\frac{\delta\pi}{2}+(1-{\delta})\ln{({{|z|}^2})}}\right)}\right)\Bigg)\text{.}
\end{align*}
Since $({z_0},{w_0})\in{H_{0}^{({\delta})}\subseteq{(\mathbb{C}\setminus\{0\})}\times\mathbb{C}}$ and since the real gradient $\nabla{r}$ vanishes nowhere, we get that $r$ is a smooth local defining function for ${{H_{t}^{({\delta})}}}$ in an open neighborhood $U\subseteq{H_{0}^{({\delta})}}$ of $({z_0},{w_0})$. So it suffices to prove that the Levi form of $r$ in direction $(-\partial{r}/\partial{w},\partial{r}/\partial{z})$ is non-negative at every point contained in $U\cap{{b{H_{t}^{({\delta})}}}}$. But if a point $(z,w)$ is contained in said intersection, then said Levi form computes to
\begin{align*}
\frac{(1-{\delta})^2}{4|z|^2}\cdot{(t-r(z,w))}=\frac{(1-{\delta})^2}{4|z|^2}\cdot{t}\text{,}
\end{align*}
which is clearly $>0$ by choice of $t$.
\end{proof}

Finally, we prove Lemma \ref{wormpaperd2locpscx}. Since our chosen domain of ``discs rotating in the $w$-plane'' is of course {\emph{not}} pseudoconvex, the assumption that $({z_0},{w_0})$ is contained in the ``good'' part of the boundary will be crucial when estimating the Levi form. We introduce some notation:

\theoremstyle{definition}
\newtheorem{notationforpscxwormpaper}[propo]{Notation}
\begin{notationforpscxwormpaper}
\label{notationforpscxwormpaper}
With our fixed choice of $\delta$, we set
\begin{align*}
\widetilde{\delta}:=\sin\left(\frac{\delta\pi}{2(1-{\delta})}\right)\text{,}
\end{align*}
and define a map $\gamma\colon\mathbb{C}\setminus\{0\}\to\mathbb{R}$ by
\begin{align*}
z\mapsto\frac{\delta\pi}{2}+(1-{\delta})\ln{({|z|}^2)}\text{.}
\end{align*}
\end{notationforpscxwormpaper}

\begin{proof}[Proof of Lemma \ref{wormpaperd2locpscx}]
Since $({z_0},{w_0})\in{H_{0}^{({\delta})}\subseteq{(\mathbb{C}\setminus\{0\})}\times\mathbb{C}}$ and since furthermore $\nabla{{\rho}_{{\delta},{\eta}}}({z_0},{w_0})\neq{0}$, the function ${\rho}_{{\delta},{\eta}}$ is a smooth local defining function for $D^{({\delta},{\eta})}$ in an open neighborhood $U\subseteq{H_{0}^{({\delta})}}$ of $({z_0},{w_0})$. Given $(\widetilde{z},\widetilde{w})\in{U}$, we denote the Levi form of ${\rho}_{{\delta},{\eta}}$ in direction $(-\partial{{\rho}_{{\delta},{\eta}}}/\partial{w},\partial{{\rho}_{{\delta},{\eta}}}/\partial{z})$ at $(\widetilde{z},\widetilde{w})$ as $L(\widetilde{z},\widetilde{w})$. So it suffices to prove that $L(z,w)$ is non-negative for every point $(z,w)$ contained in $U\cap{b{D^{({\delta},{\eta})}}}$.

To this end, let $(z,w)\in{U}\cap{b{D^{({\delta},{\eta})}}}$. Using that ${\rho}_{{\delta},\eta}(z,w)=0$, one verifies that (see Notation \ref{notationforpscxwormpaper}):
\begin{align*}
\frac{|z|^2\cdot{L(z,w)}}{(1-{\delta})^2} & =\mathcal{S}_{\eta}({\gamma}(z))\cdot\Big({-\mathcal{S}_{\eta}''({\gamma}(z))+2\operatorname{Re}((w+i\widetilde{{\delta}})\cdot\exp{(-i{\gamma}(z))})}\Big)\\
& \phantom{ =}{-}\mathcal{S}_{\eta}'({\gamma}(z))\cdot{2}\operatorname{Re}(i\cdot{(w+i\widetilde{\delta})}\cdot\exp{(-i{\gamma}(z))})\\
& \phantom{ =}+\big({\mathcal{S}_{\eta}'({\gamma}(z))}\big)^2\text{.}
\end{align*}

First, we consider the case where $\mathcal{S}_{\eta}({\gamma}(z))\leq{1}$. Since ${\rho}_{{\delta},\eta}(z,w)=0$, we have $\mathcal{S}_{\eta}({\gamma}(z))\geq{0}$. Furthermore $-\mathcal{S}_{\eta}''$ is $\geq{0}$, since $\mathcal{S}_{\eta}$ is smooth and concave. Combining this with the fact that $a^2-2ab\geq{-b^2}$ for all $a,b\in\mathbb{R}$, we get:
\begin{align*}
\frac{|z|^2\cdot{L(z,w)}}{(1-{\delta})^2} & \geq\mathcal{S}_{\eta}({\gamma}(z))\cdot{2\operatorname{Re}((w+i\widetilde{{\delta}})\cdot\exp{(-i{\gamma}(z))})}\\
& \phantom{ \geq}{-}\Big(\operatorname{Re}(i\cdot{(w+i\widetilde{\delta})}\cdot\exp{(-i{\gamma}(z))})\Big)^2\text{.}
\end{align*}
Using once again that ${\rho}_{{\delta},\eta}(z,w)=0$, we find a $\theta\in\mathbb{R}$ with the property that $(w+i\widetilde{{\delta}})\cdot\exp{(-i{\gamma}(z))}=1+\sqrt{\mathcal{S}_{\eta}({\gamma}(z))}\cdot\exp{(i{\theta})}$. Plugging in and calculating gives
\begin{align*}
\frac{|z|^2\cdot{L(z,w)}}{(1-{\delta})^2}\geq\mathcal{S}_{\eta}({\gamma}(z))\cdot\left({1-\mathcal{S}_{\eta}({\gamma}(z))}+\left({\cos{({\theta})}+\sqrt{\mathcal{S}_{\eta}({\gamma}(z))}}\right)^2\right)
\end{align*}
which is clearly $\geq{0}$, since $0\leq\mathcal{S}_{\eta}({\gamma}(z))\leq{1}$.

Now we consider the case where $\mathcal{S}_{\eta}({\gamma}(z))>{1}$. Using that $(z,w)\in{U}\subseteq{H_{0}^{({\delta})}}$, we get $\operatorname{Re}((w+i\widetilde{{\delta}})\cdot\exp{(-i{\gamma}(z))})>0$. So, since $\mathcal{S}_{\eta}$ is smooth and concave and since $\mathcal{S}_{\eta}({\gamma}(z))>1$, we immediately arrive at the following inequality:
\begin{align*}
\frac{|z|^2\cdot{L(z,w)}}{(1-{\delta})^2} & >-\mathcal{S}_{\eta}''({\gamma}(z)){-}\mathcal{S}_{\eta}'({\gamma}(z))\cdot{2}\operatorname{Re}(i\cdot{(w+i\widetilde{\delta})}\cdot\exp{(-i{\gamma}(z))})\\
& {\geq}-\mathcal{S}_{\eta}''({\gamma}(z)){-}2\left|\mathcal{S}_{\eta}'({\gamma}(z))\right|\cdot{}\left|{w+i\widetilde{\delta}}\right|\\
& {\geq}2\left|\mathcal{S}_{\eta}'({\gamma}(z))\right|\cdot{}\left(50-\left|{w+i\widetilde{\delta}}\right|\right)\text{,}
\end{align*}
where the last inequality follows by combining the properties in \ref{wormpaperessetadef} with the fact that $\mathcal{S}_{\eta}({\gamma}(z))>1$. Hence it suffices to show that $|w+i\widetilde{\delta}|<50$. But, since $\mathcal{S}_{\eta}\leq{1+\eta}<2$, that follows readily from the fact that ${\rho}_{{\delta},{\eta}}(z,w)=0$.
\end{proof}

\bibliographystyle{amsplain}
\bibliography{refspaper4}

\end{document}